\documentclass[12pt]{article} 
\usepackage{amsmath, amsthm, amssymb, enumitem, graphicx, mathtools, bbm, xcolor, xkeyval}
\usepackage[symbol,hang,flushmargin]{footmisc}
\usepackage[margin=1.5in]{geometry}
\usepackage{tikz, tkz-euclide, adjustbox}
\usepackage[colorlinks=false]{hyperref}
\usepackage[numbers]{natbib}
\usepackage{thm-restate}

\let\OLDthebibliography\thebibliography
\renewcommand\thebibliography[1]{
  \OLDthebibliography{#1}
  \setlength{\parskip}{0pt}
  \setlength{\itemsep}{0pt plus 0.3ex}
}

\newcommand{\N}{\mathbb{N}}

\newtheorem{theorem}{Theorem}
\newtheorem{corollary}[theorem]{Corollary}
\newtheorem{conjecture}[theorem]{Conjecture}
\newtheorem{lemma}[theorem]{Lemma}

\newtheorem{proposition}[theorem]{Proposition}

\title{Dense induced subgraphs of dense bipartite graphs}
\author{Rose McCarty\footnote{Department of Combinatorics and Optimization, University of Waterloo. E-mail: \href{mailto:rose.mccarty@uwaterloo.ca}{\texttt{rose.mccarty@uwaterloo.ca}} }}

\begin{document}
\maketitle

\begin{abstract}
We prove that every bipartite graph of sufficiently large average degree has either a $K_{t,t}$-subgraph or an induced subgraph of average degree at least $t$ and girth at least $6$. We conjecture that ``$6$'' can be replaced by ``$k$'', which strengthens a conjecture of Thomassen. In support of this conjecture, we show that it holds for regular graphs.
\end{abstract}

\section{Introduction}

We prove the following.

\begin{restatable}{theorem}{main}
\label{thm:main}
There is a function $f_{\ref{thm:main}}: \mathbb{N} \rightarrow \mathbb{N}$ so that for every $t\in\N$, every bipartite graph with average degree at least $f_{\ref{thm:main}}(t)$ has either a $K_{t,t}$-subgraph or an induced subgraph of average degree at least $t$ and girth at least $6$.
\end{restatable}

This refines the recent result of Kwan, Letzter, Sudakov, and Tran~\cite{KLST20} that every graph of sufficiently large average degree has either a $K_t$-subgraph or an induced, bipartite subgraph of average degree at least $t$. This resolved a conjecture of Esperet, Kang, and Thomass\'{e}~\cite{EKT19} and a question of K\"{u}hn and Osthus~\cite{KO04induced}. We conjecture that this can be pushed further, and that every bipartite graph of sufficiently large average degree has either a $K_{t,t}$-subgraph or an induced subgraph of average degree at least $t$ and girth at least $k$. This strengthens a conjecture of Thomassen~\cite{Thomassen83}. In this direction, we prove the following.

\begin{restatable}{theorem}{reg}
\label{thm:reg}
There is a function $f_{\ref{thm:reg}}:\N^3 \rightarrow \N$ so that for all $t,k,\lambda \in \N$, every bipartite graph with average degree $d \geq f_{\ref{thm:reg}}(t, k, \lambda)$ and maximum degree at most $\lambda d$ has either a $K_{t,t}$-subgraph or an induced subgraph of average degree at least $t$ and girth at least $k$.
\end{restatable}

These results extend prior work on (non-induced) subgraphs. K\"{u}hn and Osthus~\cite{KO04} proved that every bipartite graph with sufficiently large average degree has a subgraph of average degree at least $t$ and girth at least $6$. Dellamonica, Koubek, Martin, and R\"{o}dl~\cite{DKMR11} gave another proof using a theorem of F\"{u}redi~\cite{Furedi83} on uniform hypergraphs. Their proof almost works for induced subgraphs, and our main technical contribution is showing that we can find induced subgraphs where one side is regular and the other side is much smaller (Lemma~\ref{lem:main}). The girth $6$ case is currently the best known on the following beautiful conjecture of Thomassen.

\begin{conjecture}[Thomassen~\cite{Thomassen83}]
\label{conj:Th}
There is a function $f_{\ref{conj:Th}}:\mathbb{N}^2 \rightarrow \mathbb{N}$ so that for all $t,k \in \N$, every bipartite graph of average degree at least $f_{\ref{conj:Th}}(t,k)$ has a subgraph of average degree at least $t$ and girth at least $k$.
\end{conjecture}

Thomassen~\cite{Thomassen83} noted that this is false for induced subgraphs because of $K_{t,t}$. We conjecture this is the only obstruction.

\begin{conjecture}
\label{conj:new}
There is a function $f_{\ref{conj:new}}:\mathbb{N}^2 \rightarrow \mathbb{N}$ so that for all $t,k \in \N$, every bipartite graph of average degree at least $f_{\ref{conj:new}}(t,k)$ has either a $K_{t,t}$-subgraph or an induced subgraph of average degree at least $t$ and girth at least $k$.
\end{conjecture}

\noindent Conjecture~\ref{conj:new} would imply Conjecture~\ref{conj:Th} because the latter holds for regular graphs (apply the standard technique of including edges independently at random and deleting one edge from each short cycle). 

We hope that Conjecture~\ref{conj:new} can lead to new approaches to Thomassen's conjecture. In particular, both proofs of the girth $6$ case work by either finding the desired subgraph directly, or by finding a dense induced subgraph so that one side is contained in the neighbourhood of a vertex. If the latter occurs too many times, then a $K_{t,t}$-subgraph is obtained. Conjecture~\ref{conj:new} suggests that for higher girths no reduction step is needed, as we may already assume that the graph has girth at least $6$.

Following Ne\v{s}et\v{r}il, Ossona de Mendez, Rabinovich, and Siebertz~\cite{NORS19}, we say that a class of graphs is \textit{weakly sparse} if it excludes some clique and some complete bipartite graph as induced subgraphs. It is a general trend that structural properties regarding induced and non-induced subgraphs tend to coincide on weakly sparse classes. This can particularly be seen in work on induced subdivisions~\cite{KO04induced, D18} and width parameters~\cite{GW00, W19tw}. See the relevant sections of~\cite{NORS19} and~\cite{ss18survey} for summaries of results in this direction.

\section{Preliminaries}
In this section we introduce the tools needed for Theorem~\ref{thm:main} and outline its proof. In Section~\ref{sec:4-cycles} we prove Theorem~\ref{thm:main} (the girth~$6$ case) and in Section~\ref{sec:regular} we prove Theorem~\ref{thm:reg} (the approximately regular case). We conclude in Section~\ref{sec:conc} by discussing some difficulties with Thomassen’s Conjecture and proposing a weaker conjecture.

For a bipartite graph $G$ with bipartition $(A,B)$ and sets $A_1 \subseteq A$, $B_1 \subseteq B$, we write $G[A_1, B_1]$ for the induced subgraph of $G$ on vertex set $A_1 \cup B_1$. For a positive integer $r$, we say that \textit{$A_1$ is $r$-regular} if every vertex in $A_1$ has degree exactly $r$, and $A_1$ is non-empty. We write $\delta(A_1)$ and $\Delta(A_1)$ for the minimum and maximum degree of a vertex in $A_1$, respectively. If $A_1$ is empty we consider both to be zero.

Before outlining the proof of Theorem~\ref{thm:main}, we need to introduce a theorem of F\"{u}redi which is a Ramsey-type result on stars/sunflowers in uniform hypergraphs. For simplicity, we will state a slightly altered version in terms of bipartite graphs where one side is regular. 

Let $r,t\in \N$ and let $G$ be a bipartite graph with bipartition $(A,B)$ so that $A$ is $r$-regular. An \textit{$(r,t)$-partition of $(A,B)$} is a partition $B_1, \ldots, B_r$ of $B$ so that every vertex in $A$ has a neighbour in each of $B_1, \ldots, B_r$ and, for all $1 \leq i <j \leq r$, if there exist a vertex in $B_i$ and a vertex in $B_j$ with at least two common neighbours, then every vertex in $B_i$ and vertex in $B_j$ with some common neighbour in fact have at least $t$ common neighbours. In this case we say that \textit{$B_i$ and $B_j$ are neighbourly}. If such a partition exists we say that \textit{$(A,B)$ is $(r,t)$-partitionable}.

\begin{theorem}[F\"{u}redi~{\cite[Theorem~1']{Furedi83}}]
\label{thm:Furedi}
There is a function $c:\N^2 \rightarrow \mathbb{R}^+$ so that for all $r\geq t$ and every bipartite graph $G$ with bipartition $(A,B)$ so that $A$ is $r$-regular, either $G$ has a $K_{t,t}$-subgraph or there exists $A_1 \subseteq A$ with at least $c(r,t)|A|$ vertices so that $(A_1,B)$ is $(r,t)$-partitionable in $G[A_1, B]$.
\end{theorem}

The above is readily obtained from the original statement on uniform hypergraphs by considering $\mathcal{F} \coloneqq \{N(a):a \in A\}$, noting that either $G$ has $t$ vertices in $A$ with the same neighbourhood and thus a $K_{t,t}$-subgraph or $|\mathcal{F}|\geq |A|/t$. 

We will apply F\"{u}redi's Theorem to a bipartite graph with bipartition $(A,B)$ so that $A$ is $r$-regular and much larger than $|B|/c(r,t)$. If many pairs of the $(r,t)$-partition are not neighbourly, we will immediately find an induced subgraph of large average degree and girth at least $6$. Otherwise, we will find $A' \subseteq A$ and $B' \subseteq B$ so that $G[A', B']$ has large average degree and $A'$ is contained in the neighbourhood of a vertex in $B \setminus B'$. We repeat this process, and if the latter outcome occurs too many times, then we find a $K_{t,t}$-subgraph. 

In order to apply this theorem, we will find an induced subgraph where one side is regular and much larger than the other in Lemma~\ref{lem:main}, where we will use the following lemma to get most of the way there. 

\begin{lemma}[K\"{u}hn and Osthus~{\cite[Lemma~10]{KO04induced}}]
\label{lem:KO}
For all integers $r \geq 1$ and $d \geq 16$, every bipartite graph of average degree at least $8\left(4d \right)^{12r+1}$ has an induced subgraph with bipartition $(A,B)$ so that $d \leq \delta(A)\leq \Delta(A) \leq 16d$ and $|A| \geq d^{12r}|B|$.
\end{lemma}

\section{The girth 6 case}
\label{sec:4-cycles}
In this section we prove Theorem~\ref{thm:main}. The proof of Dellamonica, Koubek, Martin, and R\"{o}dl~\cite{DKMR11} for (non-induced) subgraphs actually works modulo the following lemma, but we re-create their full proof since it is not stated for induced subgraphs.

\begin{lemma}
\label{lem:main}
There is a function $D_{\ref{lem:main}}:\N^2 \rightarrow \N$ so that for all $r,\lambda \in \N$, every bipartite graph of average degree at least $D_{\ref{lem:main}}(r, \lambda)$ has an induced bipartite subgraph with bipartition $(A,B)$ so that $A$ is $r$-regular and $|A| \geq \lambda|B|$.
\end{lemma}
\begin{proof}
Set \begin{align*}
d \coloneqq \max\left( \left( \left(2\lambda\right)^{1/r} 16e\right)^{1/11}, 2r^2 \right)
\end{align*}
and define $D_{\ref{lem:main}}(r, \lambda) \coloneqq 8\left(4d \right)^{12r+1}$. By Lemma~\ref{lem:KO}, it suffices to prove that the above holds for any graph $G$ with bipartition $(A,B)$ so that $d \leq \delta(A) \leq \Delta(A) \leq 16d$ and $|A| \geq d^{12r}|B|$.

First, we claim that there is a partition $B_1, \ldots, B_r$ of $B$ into $r$ parts so that at least half of the vertices in $A$ have a neighbour in each of $B_1, \ldots, B_r$. Assign each vertex of $B$ to one of the parts $B_1, \ldots, B_r$ independently uniformly at random. Then the probability that a particular vertex in $A$ has no neighbour in $B_1$ is at most $(1-\frac{1}{r})^{2r^2}\leq e^{-2r}\leq 1/2r$. So, by the union bound, the expectation for the number of vertices in $A$ with a neighbour in each of $B_1, \ldots, B_r$ is at least $|A|/2$, and the desired partition exists. Let $A_1$ denote the set of vertices in $A$ with at least one neighbour in each of $B_1, \ldots, B_r$.

Now, form a subset $B'$ of $B$ by including each vertex independently at random with probability $p\coloneqq 1/16d$. Let $A_2$ denote the set of all vertices in $A_1$ with exactly $r$ neighbours in $B'$. We will show that, with positive probability, the graph $G[A_2, B']$ satisfies the conditions of the lemma. The probability that a vertex $a \in A_1$ has exactly one neighbour in $B_1 \cap B'$ is at least $p(1-1/16d)^{16d-1}\geq p/e$. So the probability that $a$ has exactly one neighbour in each of $B_1\cap B', \ldots, B_r\cap B'$ is at least $p^r/e^r$. Thus
\begin{align*}
    \mathbb{E}|A_2| \geq \frac{|A_1|}{\left(16ed \right)^r} \geq \frac{|A|}{2\left(16ed \right)^r} \geq \frac{d^{12r}|B|}{2\left(16ed \right)^r} \geq \lambda|B|>0,
\end{align*}
and the desired induced subgraph exists.
\end{proof}

Now we prove the main proposition, which is stated as Proposition~10 in~\cite{DKMR11}. Theorem~\ref{thm:main} follows quickly afterwards.

\begin{proposition}
\label{prop:g6}
There are functions $R, \Lambda: \N^2 \rightarrow \N$ so that for all $r,\lambda \in \N$ and every bipartite graph $G$ with bipartition $(A,B)$ so that $A$ is $R(r, \lambda)$-regular and $|A| \geq \Lambda(r, \lambda)|B|$, either\begin{itemize}
    \item[(i)] there exists an induced subgraph of $G$ of average degree at least $r$ and girth at least $6$, or
    \item[(ii)] there exist $A' \subseteq A$ and $B' \subseteq B$ so that $A'$ is $r$-regular in $G[A', B']$, $|A'|\geq \lambda|B'|$, and $A'$ is contained in the neighbourhood of a vertex in $B \setminus B'$.
\end{itemize}
\end{proposition}
\begin{proof}
Define\begin{align*}
    t \coloneqq \lambda r+1, \hspace{1em} R \coloneqq R(r, \lambda) \coloneqq t(r+1), \hspace{1em} \textrm{and} \hspace{1em} \Lambda \coloneqq \Lambda(r, \lambda) \coloneqq \left\lceil \frac{t}{c(R,t)} \right\rceil.
\end{align*}
By Theorem~\ref{thm:Furedi} of F\"{u}redi, either $G$ has a $K_{t,t}$-subgraph or there exists $A_1 \subseteq A$ with at least $c(R,t)|A|$ vertices so that $(A_1,B)$ is $(R,t)$-partitionable in $G[A_1, B]$. If $G$ has a $K_{t,t}$-subgraph, then condition \textit{(ii)} holds. So we may assume that there is an $(R,t)$-partition $B_1, \ldots, B_R$ of $(A_1,B)$.

Now, let $H$ be the graph on vertex set $\{1,2,\ldots, R\}$ where $i$ is adjacent to $j$ if $B_i$ and $B_j$ are neighbourly (that is, if there exist a vertex in $B_i$ and a vertex in $B_j$ with at least two common neighbours). Since every graph with maximum degree $\Delta$ can be coloured with $\Delta+1$ colours, either $H$ has an independent set of size $t\geq r$ or $H$ has a vertex of degree at least $r$.

If $H$ has an independent set $I\subset \{1,2,\ldots, R\}$ of size $r$, then let $B' \coloneqq \bigcup_{i \in I}B_i$. The graph $G[A_1, B']$ has girth at least $6$, every vertex in $A_1$ has $r$ neighbours in $B'$, and $|A_1| \geq c(R,t)|A|\geq c(R,t)\Lambda|B| \geq t|B|\geq |B'|$. So the average degree of $G[A_1, B']$ is at least $r$ and condition \textit{(i)} holds.

If $H$ has a vertex $k \in \{1,2,\ldots, R\}$ of degree at least $r$, then let $b \in B_k$ be any vertex with a neighbour in $A_1$ and let $J \subset \{1,2,\ldots, R\}\setminus \{k\}$ be a set of $r$ neighbours of $k$ in $H$. Let $A' \coloneqq N(b) \cap A_1$ and $B'\coloneqq \bigcup_{j \in J}N(A')\cap B_j$. Then $A'$ is $r$-regular in $G[A', B']$ and, because $B_k$ and $B_j$ are neighbourly for each $j \in J$, every vertex in $B'$ has at least $t$ neighbours in $A'$. Then, counting the edges in $G[A', B']$ in two different ways, $r|A'| \geq t|B'|$ and thus $|A'|\geq \lambda |B'|$ and condition \textit{(ii)} holds. This completes the proof.
\end{proof}

Finally we prove Theorem~\ref{thm:main}, which is restated here for convenience.

\main*

\begin{proof}
We actually prove that there are functions $R_1, \Lambda_1: \N^2 \rightarrow \N$ so that for all $s,t \in \N$, every bipartite graph with bipartition $(A,B)$ so that $A$ is $R_1(s,t)$-regular and $|A| \geq \Lambda_1(s,t)|B|$ has either an induced subgraph of average degree at least $t$ and girth at least $6$ or a $K_{s,t}$-subgraph where the side with $s$ vertices is contained in $A$. This implies the theorem by Lemma~\ref{lem:main}. 

The statement is clearly true when $s=1$, where we set $R_1(1,t)\coloneqq t$ and $\Lambda_1(1,t)\coloneqq 1$. If $s>1$, inductively define $R_1(s,t) \coloneqq R(R_1(s-1,t), \Lambda_1(s-1,t))$ and $\Lambda_1(s,t) \coloneqq \Lambda(R_1(s-1,t), \Lambda_1(s-1,t))$. Then the statement holds by Proposition~\ref{prop:g6}, noting that we always have $R_1(s,t) \geq t$.
\end{proof}

\section{The approximately regular case}
\label{sec:regular}

Now we turn to Theorem~\ref{thm:reg} on induced subgraphs of bipartite graphs which are almost regular. For $d \in \mathbb{R}^+$, $\lambda \in \N$, we say that a graph is \textit{$(d, \lambda)$-regular} if it has average degree $d$ and maximum degree at most $\lambda d$.

The standard proof that every $(d, \lambda)$-regular bipartite graph has a subgraph with large average degree and girth, for $d$ sufficiently large, goes as follows. First observe that, where $n$ is the number of vertices, such a graph has at most $n(\lambda d)^{2k-1}$ cycles of length $2k$; each edge is the first edge of at most $(\lambda d)^{2k-2}$ paths with $2k-1$ edges, and there is at most one choice for the final edge of the cycle. Now, choose each edge independently at random with probability $p$, for suitably chosen $p$, and delete one edge from each short cycle. Because there are not too many short cycles, with high probability the resulting graph has large average degree.

We will see that the analogous procedure works for induced subgraphs when a complete bipartite graph is excluded, using the following well-known theorem.

\begin{theorem}[K\"{o}vari-S\'{o}s-Tur\'{a}n~\cite{KST}]
\label{thm:KST}
For each fixed $t \in \N$, every bipartite graph with $n$ vertices on each side and no $K_{t,t}$-subgraph has $\mathcal{O}(n^{2-1/t})$ edges.
\end{theorem}

This yields the following improved bound on the number of short cycles. We could do better than having $k$ and $\lambda$ be constant, but we write things this way in order to keep the statements simple.

\begin{corollary}
\label{cor:numCycles}
For any fixed $t,k, \lambda \in \N$, every $(d, \lambda)$-regular, $n$-vertex bipartite graph with no $K_{t,t}$-subgraph has $\mathcal{O}(nd^{2k-1-1/t})$ cycles of length at most $2k$.
\end{corollary}
\begin{proof}
It suffices to show that each edge is in $\mathcal{O}(d^{2k-2-1/t})$ cycles of length exactly $2k$. Now, each edge is the first edge of at most $(\lambda d)^{2k-4}$ paths on $2k-3$ edges. For each such path $P$ with ends $x$ and $y$, the number of cycles of length $2k$ containing $P$ is at most the number of edges between the neighbourhood of $x$ and the neighbourhood of $y$. By Theorem~\ref{thm:KST}, there are at most $\mathcal{O}(d^{2-1/t})$ such edges, completing the proof.
\end{proof}

We are ready to prove Theorem~\ref{thm:reg}, which is restated here for convenience.

\reg*
\begin{proof}
Throughout the proof, we assume that $d$ is sufficiently large and $G$ is a $(d, \lambda)$-regular graph with no $K_{t,t}$-subgraph. We will find an induced subgraph of average degree at least $t$ and girth at least $2k$. For convenience, set $\epsilon \coloneqq 1/2kt$, $p \coloneqq 1/d^{1-\epsilon}$, and $n \coloneqq |V(G)|$.

Now, let $H$ be a random induced subgraph of $G$ obtained by including each vertex independently at random with probability $p$, and let $H_1$ be the graph obtained from $H$ by removing one vertex from each cycle of length at most $2k$.  Let $X_1$ be the number of cycles of $H$ of length at most $2k$, and let $X_2$ be the number of tuples $(e,C)$ so that $e$ is an edge of $H$, $C$ is a cycle of $H$ of length at most $2k$, and $e$ is incident to a vertex of $C$ but is not contained in $C$.  So $|E(H)|-|E(H_1)| \leq 2X_1+X_2$. First we bound the expectation of $2X_1+X_2$. Writing $\ell$ for the number of cycles of $G$ of length at most $2k$ and using the fact that $p^{2k} \leq dp^{2k+1}$,\begin{align*}
    \mathbb{E} [2X_1+X_2] \leq \ell\left(2p^{2k}+ 2k\lambda d p^{2k+1}\right) \leq \ell \left(4k\lambda d p^{2k+1} \right).
\end{align*}

By Corollary~\ref{cor:numCycles}, for some $c \in \N$ depending only on $t$, $k$, and $\lambda$, \begin{align*}
    \mathbb{E} [2X_1+X_2] \leq  cnd^{2k-1-1/t}\left( d p^{2k+1}\right)= cpn\left(d^{2k-1/t}p^{2k}\right)= cpn.
\end{align*} Now, $\mathbb{E} |E(H)| = p^2 dn/2 = d^\epsilon pn/2$ and $\mathbb{E}|V(H)|=pn$. So \begin{align*}
    \mathbb{E}[|E(H)|-d^\epsilon|V(H)|/4-2X_1-X_2] & \geq d^\epsilon pn/4-cpn \geq 0
\end{align*}for $d$ sufficiently large. Then $G$ has an induced subgraph of girth at least $2k$ and average degree at least $t$ when $d$ is sufficiently large.
\end{proof}

\section{Conclusion}
\label{sec:conc}
With the results presented here, there seems to be almost as much evidence for the induced subgraph version of Thomassen's Conjecture as for the original conjecture. One notable exception though is that Dellamonica and R\"{o}dl~\cite{DR11} proved Thomassen's conjecture for graphs with maximum degree at most doubly exponential in a small enough power of their average degree. With Koubek and Martin~\cite[Theorem~17]{DKMR11}, they also showed that for every $\Delta \in \N$, there are graphs with arbitrarily large average degree and no subgraph with minimum degree at least $4$ and maximum degree at most $\Delta$. So it is not possible to reduce Thomassen's Conjecture to graphs with maximum degree bounded by a function of their average degree.

The construction is random and based on a similar construction by Pyber, R\"{o}dl, and Szemer\'{e}di~\cite{PRS95}. These graphs tend to have a particular structure which seems interesting to study further. Let us say that a bipartite graph with bipartition $(A,B)$ is \textit{$r$-neighbourhood regular} if there is a partition $B_1, \ldots, B_r$ of $B$ so that every vertex in $A$ has exactly one neighbour in each of $B_1, \ldots, B_r$, and, for each $i$, all vertices in $B_i$ have the same degree $d_i \geq r$. 

It would be interesting to know if Thomassen's Conjecture holds for such neighbourhood-regular graphs, and if (in some approximate sense) such graphs can always be found as subgraphs. Thomassen's Conjecture would appear to be difficult for neighbourhood-regular graphs when $d_1 << d_2 << \ldots << d_r$, yet random constructions have few short cycles when $A$ is large. So settling this case seems like an interesting weakening of the conjecture.

\section*{Acknowledgement}
We would like to thank Jacques Verstraete for suggesting Corollary~\ref{cor:numCycles}, which greatly simplified an earlier version of this paper.

\bibliographystyle{amsplain}
\bibliography{InducedBip}

\end{document}